\documentclass[preprint]{elsarticle}

\newtheorem{theorem}{Theorem}[section]
\newtheorem{lemma}[theorem]{Lemma}

\newcommand{\N}{\mathbb{N}}
\newproof{proof}{Proof}
\newproof{prooftheorem}{Proof of Theorem 1.1}

\usepackage{amssymb}
\usepackage{fullpage}
\usepackage{amsmath}
\usepackage{amsfonts}
\usepackage{latexsym}

\begin{document}

\title{On the Growth of the Counting Function of Stanley Sequences}
\author{Richard A. Moy}
\ead{ramoy2@illinois.edu}
\address{University of Illinois, Urbana, IL 61801, USA}

\begin{abstract}
Given a finite set of nonnegative integers $A$ with no three-term arithmetic progressions, the \emph{Stanley sequence generated by} $A$, denoted as $S(A)$, is the infinite set created by beginning with $A$ and then greedily including strictly larger integers which do not introduce a three-term arithmetic progression in $S(A)$. Erd\H{o}s et al. asked whether the counting function, $S(A,x)$, of a Stanley sequence $S(A)$ satisfies $S(A,x)>x^{\frac{1}{2}-\epsilon}$ for every $\epsilon>0$ and $x>x_0(\epsilon,A)$. In this paper we answer this question in the affirmative; in fact, we prove the slightly stronger result that $S(A,x)\geq (\sqrt{2}-\epsilon)\sqrt{x}$ for $x\geq x_0(\epsilon, A)$.

\end{abstract}
\begin{keyword}
Progression-free sets; Greedy algorithm
\end{keyword}

\maketitle

\section{Introduction} Let $\N_0$ denote the set of nonnegative integers.
A subset of $\N_0$ is called \emph{$l$-free} if it contains no $l$-term
arithmetic progression. Given a finite 3-free set $A=\{ a_1,\dots, a_t\} \subset
\N_0$, $a_1<\dots<a_t$, the \emph{Stanley sequence generated by} $A$ is the infinite sequence $S(A)=\{a_1,a_2,a_3,\dots \}$ defined by the
following recursion. If $k\geq t$ and $a_1<\dots< a_k$ have been defined,
let $a_{k+1}$ be the smallest integer $a>a_k$ such that
$\{a_1,\dots,a_k\}\cup\{a\}$ is 3-free. Sequences of this type were
introduced by Odlyzko and Stanley \cite{OS} and further studied by
Erd\H os et al. \cite{E} who coined the term \emph{Stanley sequence}, and
posed several problems about the growth of such sequences. One of these
problems \cite[Problem 1, p.~123]{E} reads:

\begin{quote} Is it true that for every $\epsilon >0$ and every finite $A\subset{\N_0}$, the counting function $S(A,x)$
of $S(A)$ grows faster than $x^{\frac{1}{2}-\epsilon}$?
\end{quote}

In this paper we will answer this question in the affirmative in the following slightly stronger form:
\begin{theorem}\label{main}
Given a finite 3-free set $A\subset\N_0$, let $S(A)$ be the \emph{Stanley sequence generated by A} and $S(A,x)=|\{ s\in S(A): s\leq x\}|$ be its counting function. Then, for any $\epsilon>0$ and $x\geq x_0(\epsilon, A)$,
\[ S(A,x)\geq \left( \sqrt{2}-\epsilon \right) \sqrt{x}.\]
\end{theorem}

\section{The Proof}
Given a set $S\subset \N_0$, we define $H(S,n)$ to be the number of three-term arithmetic progressions $s_1< s_2 < n$ with $s_1,s_2\in S$, i.e. \[ H(S,n)=\# \{ (s_1,s_2): s_1,s_2\in S, s_1<s_2, n=2s_2-s_1\} .\]
\begin{lemma}\label{zerolemma}
Suppose $A=\{a_1,\dots, a_t\}\subset \N_0$, $a_1< a_2< \dots < a_t$, is a finite 3-free set and $S(A)$ is the Stanley sequence generated by $A$. If $n>\max A$, then $H(S(A),n)=0$ if and only if $n\in S(A)$.
\end{lemma}
\begin{proof}
Clearly $H(S(A),n)=0$ if $n\in S(A)$. Now, we will prove the other direction. Suppose $n\notin S(A)$ and $n>\max A$.  Let $S(A)=\{a_1,a_2,a_3,\dots\}$ with $a_1<a_2<a_3<\dots$ and let $k$ be such that $a_k<n<a_{k+1}$. Since $n\notin S(A)$ and $n>\max A$, there must be indices $i<j\leq k$ such that $a_i,a_j,n$ form a three-term arithmetic progression. Hence, $H(S(A),n)\geq 1$.
\end{proof}
\begin{lemma}\label{lemma1} We have
\[\sum_{0\leq n\leq x}H(S(A),n) \leq \frac{S(A,x)(S(A,x)-1)}{2}.\]
\end{lemma}
\begin{proof}
Let $\{a_1,\dots, a_k\}$ be the set of integers in $S(A)$ which are less than or equal to $x$.
Then
\begin{align*}
\sum_{0\leq n\leq x} {H(S(A),n)} &=\sum_{0\leq n\leq x} {\#\{ (i,j): i<j\leq k, n=2a_j-a_i\} }
\\
&\leq \#\{ (i,j): i<j\leq k\}
\\
&=\frac{k(k-1)}{2}
\\
&=\frac{S(A,x)(S(A,x)-1)}{2},
\end{align*}
since $S(A,x)=k$ by the definition of $k$.
\end{proof}

\begin{lemma}\label{lemma2} We have
\[
\sum_{\substack{{n\notin S(A)}\\ {0\leq n\leq x}}} 1 -\max A\leq \sum_{0\leq n\leq x}H(S(A),n).
\]
\end{lemma}
\begin{proof}
This inequality holds because, by Lemma \ref{zerolemma}, $H(S(A),n)\geq 1$ for all $n\notin S(A)$ satisfying $n>\max A$, and $H(S(A),n)\geq 0$ otherwise.
\end{proof}

\begin{lemma}\label{lemma3}
We have
\[ x\leq \frac{S(A,x)(S(A,x)+1)}{2}+\max A.\]
\end{lemma}
\begin{proof}
Combining the inequalities of Lemma \ref{lemma1} and Lemma \ref{lemma2} we obtain
\begin{equation*}\label{eqn2}
\sum_{\substack{{n\notin S(A)}\\{0\leq n\leq x}}} 1 -\max A \leq \frac{S(A,x)(S(A,x)-1)}{2}.
\end{equation*}
Since
\begin{align*}
\sum_{\substack{{n\notin S(A)}\\{0\leq n\leq x}}} 1 &=\lfloor x \rfloor +1-\sum_{\substack{{n\in S(A)}\\ 0\leq x\leq n}}1
\\
&\geq x-S(A,x),
\end{align*}
it follows that
\begin{equation*}
 x-S(A,x)\leq \frac{S(A,x)(S(A,x)-1)}{2}+\max A,
\end{equation*}
which implies the asserted inequality.
\end{proof}

\begin{prooftheorem}
We will proceed by contradiction. Suppose there exists $0<\epsilon<\sqrt{2}$ and some 3-free $A\subset\N_0$ such that $S(A,x)\leq (\sqrt{2}-\epsilon)\sqrt{x}$ for arbitrarily large values of $x$. At these values of $x$,
\begin{align*}
\frac{S(A,x)(S(A,x)+1)}{2} &\leq \frac{(\sqrt{2}-\epsilon)\sqrt{x}((\sqrt{2}-\epsilon)\sqrt{x}+1)}{2}
\\
&\leq\left(1-\frac{1}{\sqrt{2}}\epsilon \right)x+\sqrt{\frac{x}{2}}.
\end{align*}
By combining this with Lemma \ref{lemma3} we obtain
\begin{equation*}
x\leq \left(1-\frac{1}{\sqrt{2}}\epsilon \right)x+\sqrt{\frac{x}{2}}
\end{equation*}
for arbitrarily large values of $x$, a contradiction.
\end{prooftheorem}

\section{Concluding Remarks}
 As observed by Odlyzko and Stanley, for certain special sets A (for example, $A=\{0,3^v\}$ where $v$ is a positive integer), the Stanley sequences $S(A)$ can be explicitly described and are well understood. However, in general the behavior of these sequences seems to be rather chaotic and poorly understood.

 In their paper \cite{E}, Erd\H os et al. computed the first few hundred terms of the Stanley sequences for the sets $A=\{0,4\}, \{0,5\}, \{0,7\}, \{0,1,4\}, \{0,1,5\}$. On the basis of such data and heuristic arguments, they stated four problems about the growth of Stanley sequences. The problem resolved here, Problem 1, concerns the lower bound of the counting function $S(A,x)$; as far as we know, the other problems are still open. Problem 2 asks whether $S(A,x)\ll x^{\alpha +\epsilon}$ for some constant $\alpha<1$. This problem is still open, although an example of Odlyzko and Stanley shows that the exponent $\alpha$ cannot be smaller than $\frac{\log 2}{\log 3}$. Problem 4 asks whether there exists a set $A$ such that the terms $a_k$ of the associated Stanley sequences $S(A)$ satisfy $\lim_{k\rightarrow \infty}{(a_{k+1}-a_k)}=\infty$. This problem, too, remains open, though Savchev and Chan \cite{CS} recently resolved a related problem, Problem 6, which asks the same question for a more general class of sequences, the maximal 3-free sets.

 Stanley sequences generated by singleton sets $A=\{a\}$ are much better understood; see, for example, Gerver and Ramsey \cite{GR}.

 It would be interesting to generalize our results to ``$k$-free''
 Stanley sequences, $S_k(A)$, which are defined in the same way as $S(A)$
 but with the condition of being ``3-free'' replaced by being ``$k$-free.''

 \section{Acknowledgements}
 The author would like to extend thanks to Professor A. J. Hildebrand for his help with revising this paper and to Professor Bruce Reznick for suggesting the problem and guiding the author through the initial phase of the work.

\end{document}